\documentclass[review]{elsarticle}
\usepackage{lineno}
\usepackage[latin1]{inputenc}
\usepackage{times}
\usepackage{amssymb,mathrsfs, amsmath}
\usepackage{amsmath,amsthm}
\usepackage{amssymb}
\usepackage{latexsym}
\usepackage{amsfonts}
\usepackage{epsfig}
\usepackage{hyperref}
\usepackage{graphicx}
\usepackage{graphics}
\usepackage[all]{xy}
\usepackage{tikz-cd}
\usepackage[mathscr]{euscript}
\usepackage{mathrsfs}
\usepackage{fancyhdr}
\newtheorem{thm}{Theorem}[section]

\newtheorem{defn}{Definition}[section]
\newtheorem{prop}{Proposition}[section]

\newtheorem{lem}{Lemma}[section]

\newtheorem{rem}{Remark}[section]

\newtheorem{cor}{Corollary}[section]

\newtheorem{exmpl}{Example}[section]

\journal{Letters in Mathematical Physics}

\begin{document}
\begin{frontmatter}

\title{Cohomology and deformations  of associative superalgebras}
\author{R. B. Yadav\fnref{}\corref{mycorrespondingauthor}}
\cortext[mycorrespondingauthor]{Corresponding author}
\ead{rbyadav15@gmail.com}
\address{Sikkim University, Gangtok, Sikkim, 737102, \textsc{India}}
\begin{abstract}
In this paper we generalize to associative superalgebras Gerstenhaber's work on cohomology structure of an associative algebra. We introduce two multiplications $\cup$ and $[-,-]$ on the  cochain complex $C^*(A;A)$ of an associative superalgebra $A$.  We prove that these  multiplications induce two multiplications on $H^*(A;A)$ and make it  graded  commutative  superalgebra and   graded Lie superalgebra, respectively. Moreover, we introduce formal  deformation theory of associative superalgebras.
\end{abstract}

\begin{keyword}
\texttt{Superalgebra, module, cohomology, graded Lie superalgebra,  differential graded superalgebra, formal deformation}
\MSC[2020]   17B70 \sep 17B60 \sep 16E40   \sep 17B10 \sep 17C70  \sep   16S80
\end{keyword}
\end{frontmatter}
\section{Introduction}\label{rbsec1}

The cohomology theory of associative algebra was studied by G. Hochschild in \cite{MR11076}, \cite{MR16762}, \cite{MR22842}; and by  Murray Gerstenhaber in \cite{MR161898}. Gerstenhaber proved that there exists a cup product multiplication $\cup$ in $H^*(A;A)$ with respect to which it is a commutative graded associative algebra. It was shown that if $P$ is a two sided module over $A$, then  $H^*(A;P)$ is a two sided module over $H^*(A;A)$. He introduced a  bracket product $[-,-]$ with respect to which $H^*(A;A)$ is a graded Lie algebra.

The deformation is a tool to study a mathematical object by deforming it into a family of the same kind of objects depending on a certain parameter.  Deformation theory of algebraic structures was introduced by Gerstenhaber for rings and algebras in a series of papers \cite{MR161898}, \citep{MR171807}, \citep{MR0207793}, \citep{MR240167}. Recently,  deformation theory of  superalgebraic structures has been studied by many authors  \citep{MR4278718}, \citep{MR2654093}, \citep{MR4117244}, \citep{MR2761298}, \citep{MR871615}, \citep{MR871615}.

 Graded  algebras are of interest in physics in the context of `supersymmetries' relating particles of differing statistics. In mathematics, graded algebras are known for some time in the context of deformation theory \cite{MR0438925}. A superalgebra is a $\mathbb{Z}_2$-graded algebra $A=A_{0}\oplus A_{1}$ (that is, if $a\in A_\alpha$, $b\in A_\beta$, then $ab\in A_{\alpha+\beta}$, $\alpha,\beta\in \mathbb{Z}_2=\{0,1\}$). An associative superalgebra is a superalgebra $A=A_{0}\oplus A_{1}$ such that $(ab)c=a(bc)$, for all $a,b,c$ in $A.$

The goal of this paper is to study different algebraic structures on the cochain complex $C^*(A;A)$, the cohomology $H^*(A;A)$ of an associative superalgebra and application of this study in  the formal deformation theory of $A.$  Organization of the paper is as follows. In Section \ref{rbsec2}, we recall some basic definitions.  In Section \ref{rbsec3}, we introduce $\mathbb{Z}$-graded Lie and pre-Lie superalgebras.  In Section \ref{rbsec4}, we introduce  supermodules over superalgebras.  In Section \ref{rbsec5}, we introduce derivations of $\mathbb{Z}$-graded superalgebras. In Section \ref{rbsec6}, we discuss  cohomology of associative superalgebras.  In this section we establish a fundamental isomorphism between  $H_i^n(A;P)$ and  $H_i^{n-1}(A;C^1(A;P))$, for $n\ge 2,$ $i=0,1$ as in \cite{MR11076} for associative algebras.  In Section \ref{rbsec7}, we compute cohomology of associative superalgebras in  dimensions $0$, $1$ and $2$. In Section \ref{rbsec10}, we introduce a cup product $\cup$ for the cohomology of an associative superalgebra $A$. In this section we prove that $\{C^*(A;A),\cup\}$ is a $\mathbb{Z}$-graded associative superalgebra and coboundary map $\delta$ is a derivation on it. Also, we prove that $\{H^*(A;A),\cup\}$ is a $\mathbb{Z}$-graded associative superalgebra. In section \ref{rbsec8}, we introduce $\mathbb{Z}$-graded right pre-Lie supersystem  and discuss the $\mathbb{Z}$-graded right pre-Lie superalgebra given by it.   We show the existence of a bracket product $[-,-]$ on $C^*(A;A)$ with respect to which it is a $\mathbb{Z}$-graded Lie superalgebra and $\delta$ is a derivation  on $\{C^*(A;A),[-,-]\}$. We prove that $\{H^*(A;A),[-,-]\}$ $\mathbb{Z}$-graded Lie superalgebra. If  $P$ is a two sided module over an associative superalgebra $A$, then $H^*(A;P)$ is a  two sided module over $\{H^*(A;A),\cup\}$. In section \ref{rbsec20}, we introduce formal deformation theory of associative superalgebras. We prove that obstruction cochain  to the deformations are $3$-cocycles. We discuss equivalence of  deformations and prove that cohomology class of the  infinitesimal of a deformation depends only on its equivalence class. 
\section{Associative  superalgebra}\label{rbsec2}
In this section, we recall definitions of graded algebra, associative superalgebra,  Lie superalgebra. We give some examples of associative superalgebras.  Throughout the paper we denote a fixed field of characteristic $0$  by $K$.
\begin{defn}\label{graded vector space}
Let $\Delta$ be a commutative group and $K$ be a field. A $\Delta$-graded vector space  is a vector space $V$ over $K$ together with a family of subspaces $\{V^\alpha\}_{\alpha\in \Delta}$, indexed by $\Delta$ such that  $V=\bigoplus_{\alpha\in \Delta}V^\alpha$, the direct sum of $V^\alpha$'s.
An element $a$ in  $V^\alpha$ is called homogeneous of degree $\alpha$, we write  $\deg(a)=\alpha.$

A $\Delta$-graded algebra   over $K$ is a $\Delta$-graded vector space $E=\bigoplus_{\alpha\in\Delta}E^\alpha$ together with a bilinear map $m:E\times E\rightarrow E$ such that $m(E^\alpha\times  E^\beta)\subset E^{\alpha+\beta}$ for all $\alpha,\;\beta\in\Delta$.
 An  associative superalgebra  is a $\mathbb{Z}_2$-graded algebra $A=A_{0}\oplus A_{1}$ such that $m(m(a,b),c)=m(a,m(b,c)),$ for all $a,b,c\in A.$  A  Lie superalgebra is a $\mathbb{Z}_2$-graded algebra such that following conditions are satisfied:
 \begin{enumerate}
   \item $m(a,b)=-(-1)^{\alpha\beta}m(b,a),$
   \item  $(-1)^{\alpha\gamma}m(m(a,b),c)+(-1)^{\beta\alpha}m(m(b,c),a)+(-1)^{\gamma\beta}m(m(c,a),b)=0,$
 \end{enumerate}
  for all $a\in E_\alpha$, $b\in E_\beta$, $c\in E_\gamma$, $\alpha,\beta,\gamma\in\mathbb{Z}_2$.

In any $\mathbb{Z}_2$-graded vector space $V=V_0\oplus V_1$  we use a notation in which we replace degree $deg(a)$ of a homogeneous element $a\in V$ by `a' whenever $deg(a)$ appears in an exponent; thus, for example $(-1)^{ab}=(-1)^{deg(a)deg(b)}$.
  Let $V=V_0\oplus V_1$ and $W=W_0\oplus W_1$ be $\mathbb{Z}_2$-graded vector spaces over a field $K$. A linear map $f:V \to W$ is said to  be homogeneous of degree $\alpha$ if $f(a)\in W$ is homogeneous and  $\deg(f(a))- \deg(a)=\alpha$, for all $a\in V_\beta$, $\beta\in \{0,1\}$. We denote degree of $f$ by $\deg(f)$.
\end{defn}
\begin{exmpl}
 Let $V=V_{0}\oplus V_{1}$ be a $\mathbb{Z}_2$-graded vector space. Consider  the vector space $A$ of all homogeneous endomorphisms of $V$. Then  $A=End_{0}(V)\oplus End_{1}(V),$  where $$End_\alpha (V)=\{f\in End (V):f(V_\beta)\subset V_{\alpha+\beta}, \forall \beta\in \mathbb{Z}_2\},\;\;\alpha\in \mathbb{Z}_2.$$  $A$ is an associative superalgebra with  respect to composition operation.
\end{exmpl}
\section{Graded Lie and pre-Lie superalgebras}\label{rbsec3}
\begin{defn}
We call a $\mathbb{Z}\times\mathbb{Z}_2$-graded algebra $E=\bigoplus_{(\alpha,\beta)\in \mathbb{Z}\times \mathbb{Z}_2} E^{\alpha}_{\beta}$
  a $\mathbb{Z}$-graded superalgebra. An element $a$ in $E^{\alpha}_{\beta}$ is said to be homogeneous of degree $(\alpha,\beta)$, for all $(\alpha,\beta)\in \mathbb{Z}\times \mathbb{Z}_2.$  The anti-isomorph or opposite $A'$ of a $\mathbb{Z}$-graded superalgebra is the superalgebra which as a $K$-vector space, is identical with $A$, but in which multiplication $m'$ is given by $m'(a,b)=m(b,a)$, where $m$ is the multiplication in $A.$

For all $\alpha\in \mathbb{Z}$, $\beta\in \mathbb{Z}_2$, from here onwards whenever $\alpha+\beta$ appears in an exponent we understand it as $\alpha +\beta \mod 2.$
We call a $\mathbb{Z}$-graded superalgebra $E$ a $\mathbb{Z}$-graded associative superalgebra if
$$m(m(a,b),c)=m(m(b,c),a),$$
for all homogeneous $a,b,c\in E$. Clearly every $\mathbb{Z}$-graded associative superalgebra is an associative algebra.
We call a $\mathbb{Z}$-graded superalgebra $E$ a $\mathbb{Z}$-graded commutative superalgebra if
$$m(a,b)=(-1)^{\alpha_1\alpha_2+\beta_1\beta_2} m(b,a),$$
for all $a\in E^{\alpha_1}_{\beta_1}$, $b\in E^{\alpha_2}_{\beta_2}$. We call a $\mathbb{Z}$-graded superalgebra $E$ a $\mathbb{Z}$-graded Lie superalgebra if following conditions are satisfied
\begin{enumerate}
   \item  If $a\in E^{\alpha_1}_{\beta_1}$ and $b\in E^{\alpha_2}_{\beta_2}$ then $$m(a,b)=-(-1)^{\alpha_1\alpha_2+\beta_1\beta_2}m(b,a).$$
   \item If $a\in E^{\alpha_1}_{\beta_1}$, $b\in E^{\alpha_2}_{\beta_2}$ and $c\in E^{\alpha_3}_{\beta_3}$, then
\begin{eqnarray}\label{SASrbeqn104}
(-1)^{\alpha_1\alpha_3+\beta_1\beta_3}m(m(a,b),c)+(-1)^{\alpha_2\alpha_1+\beta_2\beta_1}m(m(b,c),a)  && \nonumber\\
+(-1)^{\alpha_3\alpha_2+\beta_3\beta_2}m(m(c,a),b) &=&0
\end{eqnarray}
\end{enumerate}
\end{defn}
\begin{defn}\label{SASrbde10}
  We call a $\mathbb{Z}$-graded superalgebra $E$ a  $\mathbb{Z}$-graded right pre-Lie superalgebra if
  \begin{eqnarray}\label{SASpreAlgb}
    &&   m(m(c,a),b)- (-1)^{\alpha_1\alpha_2+\beta_1\beta_2}m(m(c,b),a)\nonumber \\
     &=& m(c,m(a,b))-(-1)^{\alpha_1\alpha_2+\beta_1\beta_2}m(c,m(b,a)),
  \end{eqnarray}
  for all $a\in E^{\alpha_1}_{\beta_1}$, $b\in E^{\alpha_2}_{\beta_2}$ and $c\in E^{\alpha_3}_{\beta_3}$.
  An antiisomorph $A'$ of a $\mathbb{Z}$-graded right pre-Lie superalgebra is called $\mathbb{Z}$-graded left pre-Lie superalgebra.
\end{defn}
\begin{thm}\label{SASrbthm2}
Let $A$ be a $\mathbb{Z}$-graded pre-Lie superalgebra. Define  a multiplication $[-,-]:A\times A\to A$ by
\begin{equation}\label{SASrbeqn120}
 [a,b]=m(a,b)-(-1)^{\alpha_1\alpha_2+\beta_1\beta_2}m(b,a),
\end{equation}
for all $a\in A^{\alpha_1}_{\beta_1}$, $b\in A^{\alpha_2}_{\beta_2}$.
Then in the bracket product $[-,-]$ $A$ is a $\mathbb{Z}$-graded Lie superalgebra.
\end{thm}
\begin{proof}
  Clearly the bracket product $[-,-]$  satisfies $[a,b]=-(-1)^{\alpha_1\alpha_2+\beta_1\beta_2}[b,a]$,  for all $a\in A^{\alpha_1}_{\beta_1}$, $b\in A^{\alpha_2}_{\beta_2}$.  For all $a\in A^{\alpha_1}_{\beta_1}$, $b\in A^{\alpha_2}_{\beta_2}$ and $c\in A^{\alpha_3}_{\beta_3}$, by using relations \ref{SASpreAlgb}, \ref{SASrbeqn120}, we have
\begin{eqnarray}\label{SASrbeqn100}
&&(-1)^{\alpha_1\alpha_3+\beta_1\beta_3}[[a,b],c] \nonumber \\
&=&(-1)^{\alpha_1\alpha_3+\beta_1\beta_3}\{m(m(a,b),c)-(-1)^{\alpha_1\alpha_2+\beta_1\beta_2}m(m(b,a),c)\}\nonumber \\
&& (-1)^{\alpha_3\alpha_2+\beta_3\beta_2}\{-m(c,m(a,b))+(-1)^{\alpha_1\alpha_2+\beta_1\beta_2}m(c,m(b,a))\} \nonumber\\
&=&(-1)^{\alpha_1\alpha_3+\beta_1\beta_3}\{m(m(a,b),c)-(-1)^{\alpha_1\alpha_2+\beta_1\beta_2}m(m(b,a),c)\}\nonumber \\
&&(-1)^{\alpha_3\alpha_2+\beta_3\beta_2}\{-m(m(c,a),b)+(-1)^{\alpha_1\alpha_2+\beta_1\beta_2}m(m(c,b),a)\}
\end{eqnarray}
 \begin{eqnarray}\label{SASrbeqn101}
&&(-1)^{\alpha_2\alpha_1+\beta_2\beta_1}[[b,c],a] \nonumber \\ &=&(-1)^{\alpha_2\alpha_1+\beta_2\beta_1}\{m(m(b,c),a)-(-1)^{\alpha_2\alpha_3+\beta_2\beta_3}m(m(c,b),a)\}\nonumber \\
&&+(-1)^{\alpha_1\alpha_3+\beta_1\beta_3}\{-m(a,m(b,c))+(-1)^{\alpha_2\alpha_3+\beta_2\beta_3}m(a,m(c,b))\}\nonumber\\
&=&(-1)^{\alpha_2\alpha_1+\beta_2\beta_1}\{m(m(b,c),a)-(-1)^{\alpha_2\alpha_3+\beta_2\beta_3}m(m(c,b),a)\}\nonumber \\
&&+(-1)^{\alpha_1\alpha_3+\beta_1\beta_3}\{-m(m(a,b),c)+(-1)^{\alpha_2\alpha_3+\beta_2\beta_3}m(m(a,c),b)\}
\end{eqnarray}
\begin{eqnarray}\label{SASrbeqn102}
  &&(-1)^{\alpha_3\alpha_2+\beta_3\beta_2}[[c,a],b] \nonumber \\
  &=&(-1)^{\alpha_3\alpha_2+\beta_3\beta_2}\{m(m(c,a),b)-(-1)^{\alpha_3\alpha_1+\beta_3\beta_1}m(m(a,c),b)\}\nonumber \\
  &&(-1)^{\alpha_2\alpha_1+\beta_2\beta_1}\{-m(b,m(c,a))+(-1)^{\alpha_3\alpha_1+\beta_3\beta_1}m(b,m(a,c))\}\nonumber\\
  &=&(-1)^{\alpha_3\alpha_2+\beta_3\beta_2}\{m(m(c,a),b)-(-1)^{\alpha_3\alpha_1+\beta_3\beta_1}m(m(a,c),b)\}\nonumber \\
&&(-1)^{\alpha_2\alpha_1+\beta_2\beta_1}\{-m(m(b,c),a)+(-1)^{\alpha_3\alpha_1+\beta_3\beta_1}m(m(b,a),c)\}
\end{eqnarray}
From Equations \ref{SASrbeqn100}, \ref{SASrbeqn101} and  \ref{SASrbeqn102}  we conclude that Equation \ref{SASrbeqn104} holds. Therefore $A$ is a $\mathbb{Z}$-graded Lie superalgebra.
\end{proof}
\section{Supermodules over $\mathbb{Z}$-graded superalgebras}\label{rbsec4}
\begin{defn}
Let $A$ be a $\mathbb{Z}$-graded  superalgebra,  $P=\bigoplus_{(\alpha,\beta)\in \mathbb{Z}\times \mathbb{Z}_2} P^{\alpha}_{\beta}$ be a   $\mathbb{Z}\times \mathbb{Z}_2$-graded $K$-vector space and $\rho :P\times A\to P$ be a $K$-bilinear map such that $\rho(P^{\alpha_1}_{\beta_1},A^{\alpha_2}_{\beta_2})\subset P^{\alpha_1+\alpha_2}_{\beta_1+\beta_2}$ .  If $A$ is  a $\mathbb{Z}$-graded associative superalgebra, then we say that $P$ is a right supermodule over $A$, denoting $\rho(x,a)$ by $xa$ and $m(a,b)$ by $ab$, provided $x(ab)=(xa)b,$ for all $x\in P$, $a,b\in A.$ Every $\mathbb{Z}$-graded associative superalgebra is a right supermodule over itself.
If $A$ is  a $\mathbb{Z}$-graded Lie  superalgebra, then we say that $P$ is a right  supermodule over $A$, denoting $\rho(x,a)$ by $[x,a]$ and $m(a,b)$ by $[a,b]$, provided $[[x,a],b]=[x,[a,b]]+(-1)^{\alpha_3\alpha_2+\beta_3\beta_2}[[x,b],a],$ for all $x\in P^{\alpha_1}_{\beta_1}$, $a\in A^{\alpha_2}_{\beta_2},$ $b\in A^{\alpha_3}_{\beta_3}$. Every $\mathbb{Z}$-graded Lie superalgebra is a right supermodule over itself.
If $A$ is  a $\mathbb{Z}$-graded right pre-Lie superalgebra, then we say that $P$ is a right supermodule over $A$, denoting $\rho(x,a)$ by $x\circ a$ and $m(a,b)$ by $ab$, provided
$$(x\circ a)\circ b- (-1)^{\alpha_3\alpha_2+\beta_3\beta_2}(x\circ b)\circ a = x\circ (ab)-(-1)^{\alpha_3\alpha_2+\beta_3\beta_2}(x\circ (ba)),$$
for all $x\in P^{\alpha_1}_{\beta_1}$, $a\in A^{\alpha_2}_{\beta_2}$ and $b\in A^{\alpha_3}_{\beta_3}$.  Every $\mathbb{Z}$-graded right pre-Lie superalgebra is a right supermodule over itself.

We call a right supermodule $P$ over the anti-isomorph $A'$ of  a $\mathbb{Z}$-graded associative or  Lie  superalgebra $A$ as left  supermodule over it.
We say that $P$ is a (two sided) supermodule  over a $\mathbb{Z}$-graded (associative or  Lie) superalgebra $A$ if $A\oplus P$ is a $\mathbb{Z}$-graded (associative or  Lie) superalgebra such that $A$ is  subsuperalgebra of $A\oplus P$ and  $m(x,y)=0,$ for all $x,y\in P.$ Clearly, if $P$ is a (two sided) supermodule  over a $\mathbb{Z}$-graded (associative or  Lie) superalgebra $A$ then it is  a right as well as left supermodule  over  $A$.  If $P$ is a right supermodule over a $\mathbb{Z}$-graded  commutative superalgebra $A$, then if we define a left action $A\times P\to P$ of $A$  on $P$  by $ax=(-1)^{\alpha_1\alpha_2+\beta_1\beta_2}xa,$ for all  $x\in P^{\alpha_1}_{\beta_1}$, $a\in A^{\alpha_2}_{\beta_2}$, then  $P$  becomes a (two sided) supermodule over $A.$  If $P$ is a right supermodule over a $\mathbb{Z}$-graded Lie superalgebra $A$, then if we define  a left action $A\times P\to P$ of $A$  on $P$ given by $[a,x]=-(-1)^{\alpha_1\alpha_2+\beta_1\beta_2}[x,a],$ for all  $x\in P^{\alpha_1}_{\beta_1}$, $a\in A^{\alpha_2}_{\beta_2}$,  $P$  becomes  a  (two sided) supermodule over $A.$
\end{defn}
\section{Derivations of $\mathbb{Z}$-graded superalgebras}\label{rbsec5}
\begin{defn}\label{SASrbdefn110}
Let $A=\bigoplus_{(\alpha,\beta)\in \mathbb{Z}\times \mathbb{Z}_2} A^{\alpha}_{\beta}$ be a $\mathbb{Z}$-graded superalgebra. A $K$-linear map $D:A\to A$ is called left derivation of degree $(\alpha,\beta)$ of $A$ if $D$ is homogeneous of  degree $(\alpha,\beta)$ and
\begin{equation}\label{SASrbeqn110}
  D(ab)=(Da)b+(-1)^{\alpha\alpha_1+\beta\beta_1}a(Db),
\end{equation}
  for all $a\in A^{\alpha_1}_{\beta_1}$ and $b\in A^{\alpha_2}_{\beta_2}.$ A $K$-linear map $D:A\to A$ is called a right derivation of degree $(\alpha,\beta)$ of $A$ if $D$ is homogeneous of degree $(\alpha,\beta)$ and
\begin{equation}\label{SASrbeqn111}
D(ab)=(-1)^{\alpha\alpha_2+\beta\beta_2}(Da)b+a(Db),
\end{equation}
for all $a\in A^{\alpha_1}_{\beta_1}$ and $b\in A^{\alpha_2}_{\beta_2}.$
Let $\mathcal{D}=\bigoplus_{(\alpha,\beta)\in \mathbb{Z}\times \mathbb{Z}_2} \mathcal{D}^{\alpha}_{\beta}$ be the vector space obtained by taking  direct sum of the vector spaces $\mathcal{D}^{\alpha}_{\beta}$ of  right derivations of $A$ of degree $(\alpha,\beta)$, $(\alpha,\beta)\in \mathbb{Z}\times \mathbb{Z}_2$. For $D_1\in \mathcal{D}^{\alpha_1}_{\beta_1}$, $D_2\in \mathcal{D}^{\alpha_2}_{\beta_2}$, if we define $$[D_1,D_2]=D_1D_2-(-1)^{\alpha_1\alpha_2+\beta_1\beta_2}D_2D_1,$$ then it can be easily verified that $[D_1,D_2]$ is a right derivation of $A$ of degree $(\alpha_1+\alpha_2,\beta_1+\beta_2)$ and  with this multiplication $\mathcal{D}$ is a $\mathbb{Z}$-graded Lie  superalgebra. Similar statement can be given if $\mathcal{D}=\bigoplus_{(\alpha,\beta)\in \mathbb{Z}\times \mathbb{Z}_2} \mathcal{D}^{\alpha}_{\beta}$ is  the vector space obtained by taking  direct sum of the  vector spaces $\mathcal{D}^{\alpha}_{\beta}$ of  left derivations of $A$ of degree $(\alpha,\beta)$, $(\alpha,\beta)\in \mathbb{Z}\times \mathbb{Z}_2$.  If $A=\bigoplus_{(\alpha,\beta)\in \mathbb{Z}\times \mathbb{Z}_2} A^{\alpha}_{\beta}$ is a $\mathbb{Z}$-graded associative  superalgebra and $ a\in A^{\alpha}_{\beta}$, then if we define two $K$-linear maps  $D^a_1,D^a_2:A\to A$ by
\begin{equation}\label{SASrbeqn112}
  D^a_1b=ab-(-1)^{\alpha\alpha'+\beta\beta'}ba
\end{equation}
\begin{equation}\label{SASrbeqn113}
  D^a_2b=ba-(-1)^{\alpha\alpha'+\beta\beta'}ab,
\end{equation}
  for all $b\in A^{\alpha'}_{\beta'}$, then $D^a_1$ and  $D^a_2$ are left and right derivations of $A$, respectively of degree $(\alpha,\beta).$  Similarly if  $A=\bigoplus_{(\alpha,\beta)\in \mathbb{Z}\times \mathbb{Z}_2} A^{\alpha}_{\beta}$ is a $\mathbb{Z}$-graded Lie  superalgebra and $ a\in A^{\alpha}_{\beta}$, then if we define
  \begin{equation}\label{SASrbeqn114}
 D^a_1b=[a,b]
\end{equation}
\begin{equation}\label{SASrbeqn115}
  D^a_2b=[b,a],
\end{equation}
for all $b\in A^{\alpha'}_{\beta'}$, then $D^a_1$ and  $D^a_2$ are left and right derivations of $A$, respectively of degree $(\alpha,\beta).$ $D^a_1$ and  $D^a_2$ are called inner derivations of $A$ induced by $a$.
\end{defn}
Let $A$ be a $\mathbb{Z}$-graded superalgebra and $P$ be a (two sided) module over $A.$ A $K$-linear map $D:A\to P$ is called a left  derivation, respectively a right derivation, of degree $\alpha$ of $A$ into $P$ if \ref{SASrbeqn110}, respectively \ref{SASrbeqn111} holds, for all $a\in A^{\alpha_1}_{\beta_1}$ and $b\in A^{\alpha_2}_{\beta_2}.$ If A is $\mathbb{Z}$-graded  associative or Lie superalgebra, then we can define left and right inner derivations   using relations \ref{SASrbeqn112}, \ref{SASrbeqn113},\ref{SASrbeqn114}, \ref{SASrbeqn115}. In this case, we choose $ a\in P^{\alpha}_{\beta}$.
\begin{defn}
 We call a  $\mathbb{Z}$-graded superalgebra $A$ as a differential graded superalgebra if it is equipped with  a (right or left) derivation $D:A\to A$ of degree $(1,0)$ such that $d^2=0.$
\end{defn}
\section{Cohomology of associative  superalgebras}\label{rbsec6}
Let $V=V_0\oplus V_1$, $W=W_0\oplus W_1$ be $\mathbb{Z}_2$-graded $K$-vector spaces.  An n-linear map $f:V \underset{n\; times}{\underbrace{\times\cdots\times}}V\to W$ is said to  be homogeneous of degree $\alpha$ if for all homogeneous  $x_i\in V$, $1\le i\le n$, $f(x_1,\cdots, x_n)$  is a homogeneous element in $W$  and  $\deg(f(x_1,\cdots, x_n))-\sum_{i=1}^{n}\deg(x_i))=\alpha$.  We denote the  degree of a homogeneous $f$ by $\deg(f)$.   We use a notation in which we replace degree $\deg(f)$  by $`f'$ whenever $\deg(f)$ appears in an exponent; thus, for example $(-1)^{\deg(f)}=(-1)^f$. For each $n\ge 0$, we define  a $K$-vector space  $C^{n}(V;W)$ as follows:  For $n\ge 1$, $C^{n}(V;W)$ consists of   homogeneous $n$-linear maps  $ f :V \underset{n\; times}{\underbrace{\times\cdots\times}}V\to W$,
and $C^0(V;W)=W$. Clearly,  $C^{n}(V;W)=C_0^{n}(V;W)\oplus C_1^{n}(V;W)$, where $C_i^{n}(V;W)$  is the $K$-vector subspace of $C^{n}(A;P)$ consisting of  elements of degree $i$ with  $i=0,1$.

 Let $A=A_0\oplus A_1$ be an associative  superalgebra and $P=P_0\oplus P_1$ be a (two sided) supermodule over $A$.
We define two $K$-bilinear maps   $$A\times C^{1}(A;P)\to C^{1}(A;P)\;\text{ and}\;  C^{1}(A;P)\times A\to C^{1}(A;P)$$ (we use same symbol $*$ for both the maps and differentiate them from context) by
\begin{eqnarray}\label{rbSLM1}
(a*f)(a_1)&=&af(a_1),
\end{eqnarray}
\begin{eqnarray}\label{rbSLM2}
(f*a)(a_1) &=&f(aa_1)-f(a)a_1,
\end{eqnarray}
for all $a,a_1\in A$, $f\in C^{1}(A;P)$. We have following proposition:
\begin{prop}
$C^{1}(A;P)$ is a  (two sided) supermodule over $A$.
\end{prop}
\begin{proof}
  Proof is a direct consequence of the two actions of $A$ on $ C^{1}(A;P)$  given by relations \ref{rbSLM1}, \ref{rbSLM2} and the   definition of supermodule.
\end{proof}
We define a  $K$-linear map $\delta^n:C^{n}(A;P)\to C^{n+1}(A;P)$  by
\begin{eqnarray}\label{rbSAS8}
\delta^n f(x_1,\cdots, x_{n+1})&=&(-1)^{x_1f}x_1.f(x_2,\cdots,x_{n+1})\nonumber\\
      &&+ \sum_{i=1}^{n}(-1)^if(x_1,\cdots,x_i.x_{i+1},\cdots,x_{n+1}) \nonumber\\
      &&+(-1)^{n+1}f(x_1,\cdots,\cdots,x_{n}).x_{n+1},
  \end{eqnarray}  for all $f$ in  $C^{n}(A;P)$, $n\ge 1$, and $\delta^0f(x_1)=(-1)^{x_1f}x_1.f-f.x_1$, for all $f$ in $C^0(A;P)=P$. Clearly, for each
  $f \in C^{n}(A;P)$, $n\ge 0,$  $\deg(\delta f)=\deg(f).$
 For each $f\in C^{n}(A;P)$, $n>0$, we define $f_{n-1}\in C^{n-1}(A;C^1(A;P))$,  by
$$ f_{n-1}(a_1,\cdots,a_{n-1})(a_{n})=f(a_1,\cdots,a_n).$$

\begin{thm}
$\delta\delta=0$, that is, $(C^{\ast}(A;P),\delta)$ is a cochain complex.
  \end{thm}
  \begin{proof}
    For $f\in C^{0}(A;P),$ we have
    \begin{eqnarray*}
\delta\delta f(x_1,x_2) &=& (-1)^{x_1f}x_1\delta f(x_2)-\delta f(x_1.x_2)+\delta f(x_1).x_2\\
                        &=& (-1)^{x_1f}x_1.((-1)^{x_2f}x_2.f -f.x_2)-(-1)^{x_2f+x_1f}(x_1.x_2).f\nonumber\\
                        &&+f.(x_1.x_2)+(-1)^{x_1f}(x_1.f).x_2- (f.x_1).x_2\\
                        &=&0.
    \end{eqnarray*}
For $f\in C^{n}(A;P)$, $n\ge 1$, we have
\begin{eqnarray}
(\delta f)_{n}(x_1,\cdots, x_{n})(x_{n+1})&=&\delta f(x_1,\cdots, x_{n+1})\nonumber\\
 &=&(-1)^{x_1f}x_1.f(x_2,\cdots,x_{n+1})\nonumber\\
     &&+\sum_{i=1}^{n}(-1)^if(x_1,\cdots,x_i.x_{i+1},\cdots,x_{n+1})\nonumber \\
     &&+(-1)^{n+1}f(x_1,\cdots,x_{n}).x_{n+1}\nonumber\\
     &=&(-1)^{x_1f}(x_1*f_{n-1}(x_2,\cdots,x_{n})(x_{n+1})\nonumber\\
     &&+ \sum_{i=1}^{n-1}(-1)^if_{n-1}(x_1,\cdots,x_i.x_{i+1},\cdots,x_{n})(x_{n+1})\nonumber \\
     &&+(-1)^{n}(f_{n-1}(x_1,\cdots,x_{n-1})*x_{n})x_{n+1}\nonumber\\
     &=&\delta f_{n-1}(x_1,\cdots,x_{n})(x_{n+1})
\end{eqnarray}
Thus $(\delta f)_{n}=\delta( f_{n-1})$. This implies that for all $f\in C^{n}(A;P)$, $n \ge 1,$
\begin{equation}\label{AssSup1}
(\delta\delta f)_{n+1}=\delta ((\delta f)_{n})=\delta \delta( f_{n-1})
\end{equation}
Assume that $\delta\delta f=0$ holds, for all $f\in C^{q}(A;P)$, where $P$ is an arbitrary supermodule over $A$, $0\le q\le n$. Using Equation \ref{AssSup1}, for  $f\in C^{n+1}(A;P)$ we have $\delta \delta (f_{n})=(\delta\delta f)_{n+2}$.  By induction hypothesis $\delta \delta (f_{n})=0.$ This implies that $(\delta\delta f)_{n+2}=0.$ Since $f=0$ if and only if $f_{n-1}=0,$ for all $f\in C^{n}(A;P)$, $n\ge 1$, we conclude that $\delta \delta f=0.$ So, by using mathematical induction we conclude that $\delta\delta=0.$
\end{proof}
We denote $\ker(\delta^n)$ by $Z^n(A;P)$ and image of $\delta^{n-1}$ by $B^n(A;P)$.
We call  the $n$-th  cohomology $Z^n(A;P)/B^n(A;P)$ of the cochain complex $(C^{*}(A;P),\delta)$ as the  $n$-th  cohomology of $A$ with coefficients in $P$ and  denote it by $H^{n}(A;P)$. Since $A$ is a supermodule over itself. So we can consider  cohomology groups $H^{n}(A;A)$. We call $H^{n}(A;A)$ as the $n$-th  cohomology group of $A$.
We  have $$Z^n(A;P)=Z_0^n(A;P)\oplus Z_1^n(A;P),  \;B^n(A;P)=B_0^n(A;P)\oplus B_1^n(A;P),$$  $Z_i^n(A;P)=\{f\in Z^n(A;P): \deg(f)=i \}$, $B_i^n(A;P)=\{f\in B^n(A;P): \deg(f)=i\}$ are vector subspaces of  $Z^n(A;P)$ and $B^n(A;P)$,respectively,  $i=0,1$. Since boundary map  $\delta^n:C^{n}(A;P)\to C^{n+1}(A;P)$ is homogeneous of degree $0$, we conclude that $H^{n}(A;P)$ is $\mathbb{Z}_2$-graded and $$H^{n}(A;P) \cong H_0^{n}(A;P)\oplus H_1^{n}(A;P),$$ where $H_i^{n}(A;P)= Z_i^n(A;P)/B_i^n(A;P)$, $i=0,1$.
\begin{thm}
  For $n\ge 2,$ $H_i^n(A;P)\cong H_i^{n-1}(A;C^1(A;P))$, $i=0,1$.
\end{thm}
\begin{proof}
Clearly, for $i=0,1$, the mapping $f \mapsto f_{n-1}$  is an isomorphism from $C_i^n(A;P)$ onto $C_i^{n-1}(A;C^1(A;P))$. Since  $(\delta f)_n=\delta (f_{n-1})$,  $Z_i^n(A;P)$ and $B_i^n(A;P)$ are mapped by the mapping $f \mapsto f_{n-1}$ onto $Z_i^{n-1}(A;C^1(A;P))$ and $B_i^{n-1}(A;C^1(A;P)),$ respectively. Hence, for $i=0,1$, the mapping $f\mapsto f_{n-1}$ induces an isomorphism from $H_i^n(A;P)$ onto $H_i^{n-1}(A;C^1(A;P))$.
\end{proof}
\section{Cohomology of associative superalgebras in low dimensions}\label{rbsec7}
Let $A=A_0\oplus A_1$ be an associative superalgebra and $P=P_0\oplus P_1$ be a supermodule over $A.$ For any  $m\in C^0_i(A;P)=P_i$, $f\in C^1_i(A;P)$ and $g\in C^2_i(A;P)$, $i=0,1$
\begin{equation}\label{SLCOH1}
  \delta^0 m(x)=(-1)^{mx}x.m -m.x,
\end{equation}
\begin{equation}\label{SLCOH2}
 \delta^1 f(x_1,x_2)=-f(x_1.x_2)+ (-1)^{x_1f}x_1.f(x_2)+f(x_1).x_2,
\end{equation}
\begin{eqnarray}\label{SLCOH3}
   \delta^2 g(x_1,x_2,x_3)&=& -g(x_1.x_2,x_3)+g(x_1,x_2.x_3)\nonumber\\
   &&+(-1)^{x_1g}x_1.g(x_2,x_3)-g(x_1,x_2).x_3.
\end{eqnarray}
 We denote the  set $\{m\in P_i|m.x=(-1)^{mx}x.m, \text{for all homogeneous}\; x\in A\}$  by $Com_{P_i}A$.  We have \begin{eqnarray*}
          H_i^0(A;P) &=& \{m\in P_i|(-1)^{mx}x.m-m.x=0,\;\text{for all homogeneous}\; x\in A\} \\
          &=& Com_{P_i}A.
         \end{eqnarray*}
   For every $m\in P_i$ the map $x\mapsto (-1)^{ix}x.m- m.x$ is  a left inner derivation  of   $A$ into $P$ of degree $i$ and induced by $-m$. We denote the vector spaces of left derivations and  left inner derivations of degree $i$ of  $A$ into $P$ by $Der_i(A;P)$ and $Der_i^{Inn}(A;P)$ respectively. By using \ref{SLCOH1}, \ref{SLCOH2} we have $$H_i^1(A;P)=Der_i(A;P)/Der_i^{Inn}(A;P).$$
Let $A$ be an associative superalgebra and $P$ be a supermodule over $A.$ We regard $P$ as an associative superalgebra with  trivial product  $xy=0,$ for all $x,y\in P$.  We define extension of $A$ by $P$ of degree $0$ to be an exact sequence
\[\xymatrix{0\ar[r]& P\ar[r]^i &\mathcal{E}\ar[r]^\pi &A\ar[r] &0 }\tag{*}\]
of associative superalgebras such that $\deg(i)=0= \deg(\pi)$ and
\begin{equation}\label{SASrbeqn116}
  x.i(m)=\pi(x).m,\; i(m).x=m.\pi(x),
\end{equation}
for all homogeneous $x\in \mathcal{E}, $ $m\in P.$
 The exact sequence $(*)$ regarded  as a sequence of $K$-vector spaces  splits. Therefore without any loss of generality we may assume that  $\mathcal{E}$ as a  $K$-vector space coincides with the direct sum $A\oplus P$ and that $i(m)=(0,m),$ $\pi(x,m)=x.$ Hence relation \ref{SASrbeqn116} is meaningful. Thus we have  $\mathcal{E}=\mathcal{E}_0\oplus \mathcal{E}_1,$ where $\mathcal{E}_0=A_0\oplus P_0$, $\mathcal{E}_1=A_1\oplus P_1.$ The multiplication in $\mathcal{E}=A\oplus P$ has then necessarily the form $$ (0,m_1).(0,m_2)   = (0,0),\;  (x_1,0).(0,m_1) = (0,x_1.m_1),$$
 $$(0,m_2).(x_2,0) = (0,m_2.x_2),\; (x_1,0).(x_2,0) = (x_1.x_2,h(x_1,x_2)),$$  for some $h\in C_0^2(A;P)$, for all  homogeneous $x_1,x_2\in A$, $m_1,m_2\in P.$
Thus, in general, we have
 \begin{equation}\label{SLCOH7}
(x,m).(y,n)=(x.y,x.n+m.y+h(x,y)),
\end{equation}
for all  homogeneous $(x,m)$, $(y,n)$ in $\mathcal{E}=A\oplus P.$\\
 Conversely, let $h:A\times A\to P$ be a bilinear homogeneous map of degree $0$. For homogeneous $(x,m)$, $(y,n)$ in $\mathcal{E}=A\oplus P$ we define multiplication by Equation \ref{SLCOH7}.
For homogeneous $(x,m)$, $(y,n)$ and $(z,p)$ in $\mathcal{E}$ we have
\begin{eqnarray}\label{SLCOH4}
((x,m).(y,n)).(z,p)&=&((x.y).z,(x.y).p+(x.n).z+(m.y).z\nonumber\\
                      &&+h(x,y).z+h(x.y,z)
\end{eqnarray}
\begin{eqnarray}\label{SLCOH5}
(x,m).((y,n).(z,p))&=&(x.(y.z), x.(y.p)+x.(n.z)+m.(y.z)\nonumber\\
                     &&+x.h(y,z)+h(x,y.z)
\end{eqnarray}
 From Equations \ref{SLCOH4}, \ref{SLCOH5},  we conclude that $\mathcal{E}=A\oplus P$ is an associative superalgebra  with product given by Equation \ref{SLCOH7}  if and only if $\delta^2 h=0.$
 We denote the associative superalgebra given by Equation \ref{SLCOH7} using notation $\mathcal{E}_h$. Thus for every cocycle $h$ in  $C_0^2(A;P)$ there exists an extension
 \label{SLCOH8}
   \[E_h:\xymatrix{0\ar[r]& P\ar[r]^i &\mathcal{E}_h\ar[r]^\pi &A\ar[r] &0 }\]
of $A$ by $P$ of degree $0$, where $i$ and $\pi$ are inclusion and projection maps, that is, $i(m)=(0,m),$ $\pi(x,m)=x$.
 We say that two extensions
  \[\xymatrix{0\ar[r]& P\ar[r] &\mathcal{E}^i\ar[r] &A\ar[r] &0 } \;(i=1,2)\]
  of $A$ by $P$ of degree $0$ are equivalent if there is an associative superalgebra isomorphism $\psi:\mathcal{E}^1\to \mathcal{E}^2$ of degree $0$ such that following diagram commutes:
\[
\xymatrix{
  0 \ar[r] & P \ar[d]_{Id_P} \ar[r]^-{} & \mathcal{E}^1 \ar[d]_-{\psi} \ar[r]^-{} & A \ar[d]^-{Id_A} \ar[r] & 0 \\
  0 \ar[r] & P \ar[r]_-{} & \mathcal{E}^2 \ar[r]_-{} & A \ar[r] & 0
}
\tag{**}\]
We use $F(A,P)$ to denote the set of all equivalence classes of extensions of   $A$ by $P$ of degree $0$. Equation \ref{SLCOH7} defines a mapping of $Z_0^2(A;P)$
 onto $F(A,P)$. If for $h,h'\in Z_0^2(A;P)$ $E_h$ is equivalent to $E_{h'}$, then commutativity of diagram $(**)$ is equivalent to $$ \psi(x,m)=(x,m+f(x)),$$  for some $f\in C_0^1(A;P)$.
 We have
 \begin{eqnarray}
   \psi((x_1,m_1).(x_2,m_2)) &=& \psi(x_1.x_2,x_1.m_2+m_1.x_2+h(x_1,x_2))\nonumber \\
    &=& (x_1.x_2,x_1.m_2+m_1.x_2+h(x_1,x_2)+f(x_1.x_2)),\nonumber\\
 \end{eqnarray}
  \begin{eqnarray}
    \psi(x_1,m_1).\psi(x_2,m_2)&=& (x_1,m_1+f(x_1)).(x_2,m_2+f(x_2)) \nonumber\\
     &=& (x_1.x_2,x_1.(m_2+f(x_2))+(m_1+f(x_1)).x_2 +h'(x_1,x_2)).\nonumber\\
  \end{eqnarray}
 Since $\psi((x_1,m_1).(x_2,m_2))=\psi(x_1,m_1).\psi(x_2,m_2)$, we have
 \begin{eqnarray}
h(x_1,x_2)-h'(x_1,x_2) &=& - f(x_1.x_2)+x_1.f(x_2)+f(x_1).x_2\nonumber\\
    &=& \delta^1(f)(x_1,x_2)
 \end{eqnarray}
Thus  two extensions $E_h$ and $E_{h'}$ of degree $0$ are equivalent if and only if there exists some $f\in C_0^1(A;P)$ such that $\delta^1f = h-h'$. We thus have following theorem:
\begin{thm}
The set $F(A,P)$ of  all equivalence classes of  extensions of $A$ by $P$ of degree $0$ is in one to one correspondence with the cohomology group $H_0^2(A;P)$. This correspondence $\omega :H_0^2(A;P)\to F(A,P)$ is obtained  by assigning to each cocycle $h\in Z_0^2(A;P)$, the extension given by multiplication \ref{SLCOH7}.
\end{thm}
\section{Cup product for $H^*(A;A)$}\label{rbsec10}
\begin{defn}
Let $P$ be a (two sided) supermodule over an associative superalgebra $A$. Also, assume that $P$ has an structure associative superalgebra, in particular we may take $P=A.$ We define a multiplication $\cup$ on $C^*(A;P)=\bigoplus C^n(A;P)$ by defining  $\cup: C^m(A;P)\times C^n(A;P)\to C^{m+n}(A;P)$  by
\begin{eqnarray}
     && f\cup g(a_1,\cdots,a_m,b_1,\cdots,b_n)\nonumber \\
     &=& (-1)^{\tilde{g}(a_1+\cdots+a_m)}f(a_1,\cdots,a_m)g(b_1,\cdots,b_n),
\end{eqnarray}
for all $f\in  C^{m}_{\tilde{f}}(A;P)$, $g\in  C^{n}_{\tilde{g}}(A;P)$, for all $m,n\ge 0$. If we put $C^n(A;P)=0$ for $n<0,$ then  $C^*(A;P)$ is a $\mathbb{Z}$-graded associative superalgebra with respect to the multiplication $\cup$.
\end{defn}
As a direct consequence of definitions of coboundary map $\delta$  \ref{rbSAS8} and $\cup$  we have following proposition:
\begin{prop}\label{rbSASprop7}
  For $f\in  C^{m}_{\tilde{f}}(A;P)$, $g\in  C^{n}_{\tilde{g}}(A;P)$,
\begin{equation}\label{rbSAS7}
    \delta(f\cup g)=\delta f\cup g+(-1)^mf\cup \delta g.
\end{equation}
\end{prop}
From Proposition \ref{rbSASprop7}, we conclude that $\delta$ is a derivation of $\{C^*(A;P),\cup\}$ of degree $(1,0)$. Hence the cochain complex $C^*(A;A)$ of an associative super algebra is a differential graded  associative superalgebra  with respect to the multiplication $\cup$.  Next, we have some observations as  following theorem:
\begin{thm}
  \begin{itemize}
    \item[(i)]  $\{Z^*(A;P),\cup\}$ is a subsuperalgebra of $\{C^*(A;P),\cup\}$, where $Z^*(A;P)=\bigoplus Z^n(A;P)$;
    \item[(ii)] If one of $f$ and $g$ is in  $  B^{m}(A;P)$ and other in  $  Z^{n}(A;P)$, then   $f\cup g\in  B^{m+n}(A;P)$, that is, $B^*(A;P)=\bigoplus B^{n}(A;P)$ is a (two sided) ideal of $Z^*(A;P)$;
    \item[(iii)] $H^*(A;P)=\bigoplus H^n(A;P)$ is a $\mathbb{Z}$-graded associative superalgebra with respect to multiplication $\cup$ (called as cup product) defined by $$(f+B^{m}(A;P))\cup (g+B^n(A;P))=(f\cup g)+B^{m+n}(A;P),$$  for all $f\in Z^{m}(A;P)$, $g\in Z^{n}(A;P)$.
  \end{itemize}
\end{thm}

Let $A$ be an associative superalgebra and $P$ be a (two sided) supermodule over $A$. If we define two actions of $C^*(A;A)=\oplus C^n(A;A)$ on $C^*(A;P)=\oplus C^n(A;P)$ (we use same symbol $\cup$ for both the actions and differentiate them from context) by $$f\cup g(a_1,\cdots,a_m,b_1,\cdots,b_n )=(-1)^{\tilde{g}(a_1+\cdots+a_m)}f(a_1,\cdots,a_m)g(b_1,\cdots,b_n),$$
for all $f\in  C^{m}_{\tilde{f}}(A;A)$, $g\in  C^{n}_{\tilde{g}}(A;P)$  or $f\in  C^{m}_{\tilde{f}}(A;P)$, $g\in  C^{n}_{\tilde{g}}(A;A)$. With  these two actions $C^*(A;P)$ is a (two sided) supermodule over $C^*(A;A)$

It is clear that Equation $\ref{rbSAS7}$ holds also when either $f\in C^m(A;P)$, $g\in C^n(A;A)$ or $f\in C^m(A;A)$, $g\in C^n(A;P)$. This implies that $Z^m(A;A)\cup Z^n(A;P)$ and $Z^m(A;P)\cup Z^n(A;A)$ are contained in $Z^{m+n}(A;P)$, and $Z^m(A;A)\cup B^n(A;P)$ and $B^m(A;A)\cup Z^n(A;P)$ are contained in $B^{m+n}(A;P)$. Hence we conclude that $H^*(A;P)$ is a (two sided) supermodule over $\{H^*(A;A),\cup\}$.
\section{Bracket product for $H^*(A;A)$}\label{rbsec8}
\begin{defn}
A right pre-Lie super system $\{V_m,\circ_i\}$ is defined as a sequence $$\cdots,V_{-1},V_0,V_1,V_2\cdots,V_n,\cdots$$ of $\mathbb{Z}_2$-graded $K$-vector spaces $V_m=V^m_0\oplus V^m_1$ together with an assignment for every triple of integers $m,n,i\ge 0 $ with $i\le m$, of a homogeneous $K$-bilinear map $\circ_i=\circ_i(m,n)$ of $V_m\times V_n$ into $V_{m+n}$ of degree $ 0$ such that following conditions are satisfied:

  \begin{equation}\label{LS10}
     (f\circ_ig)\circ_jh=
     \begin{cases}
       (-1)^{\tilde{g} \tilde{h}}(f\circ_jh)\circ_{i+p}g, & \mbox{if } 0\le j\le i-1 \\
        f\circ_i(g\circ_{j-i}h), & \mbox{if } i\le j\le n+i
     \end{cases}
\end{equation}
for all  $f\in V^m_{\tilde{f}},$ $g\in V^n_{\tilde{g}}$ and $h\in V^p_{\tilde{h}}$. Here $\circ_i(f, g)=f\circ_ig$.
From the first case of Equation \ref{LS10}, we have
\begin{equation}\label{LS11}
(f\circ_jh)\circ_{i+p}g= (-1)^{\tilde{g} \tilde{h}}(f\circ_{i+p}g)\circ_{j+n}h,\;\text{if}\; 0\le i+p\le j-1
\end{equation}
\end{defn}
In  Equation \ref{LS11}, replacing $i+p$ by $i$ and $j+n$ by $j$ we have
\begin{equation}\label{LS12}
(f\circ_ig)\circ_jh=(-1)^{\tilde{g} \tilde{h}}(f\circ_{j-n}h)\circ_{i}g,\;\text{if}  n+i+1\le j\le m+n,
\end{equation}
\begin{exmpl}
  Let $A=A_0\oplus A_1$ be an associative superalgebra. Put $V_0=A$, and $V_m=0,$ for all $m\in \mathbb{Z}$ with $m\ne 0.$  We define $\circ_i=\circ_i(m,n):V_m\times V_n\to V_{m+n}$ by $$a\circ_i b= \begin{cases}
       ab, & \mbox{if } i=m=n=0 \\
       0, & \mbox{otherwise}.
     \end{cases}$$
Clearly, $\{V_m,\circ_i\}$ is a right pre-Lie supersystem.
\end{exmpl}
\begin{exmpl}\label{SASrbe1}
  Let $U=U_0\oplus V_1$, $W=W_0\oplus W_1$ $\mathbb{Z}_2$-graded $K$-vector spaces. Let   $\phi :W\to U$ be a homogeneous $K$-linear map of degree $0$.  Put  $$V_m=\begin{cases}
    C^{m+1}(U;W), & \mbox{if } m\in\mathbb{Z}, m\ge -1 \\
    0, & \mbox{otherwise}.
  \end{cases}$$ For integers $m,n,i\ge 0,$  $i\le m$ we define $\circ_i:V_m\times V_n\to V_{m+n}$ by
  \begin{eqnarray}
   && f\circ_ig(x_1,\cdots,x_i,y_1,\cdots,y_{n+1},x_{i+2},\cdots, x_{m+1}) \nonumber \\
   &=& (-1)^{\tilde{g}(x_1+\cdots+x_i)} f(x_1,\cdots,x_i,\phi g(y_1,\cdots,y_{n+1}), x_{i+2},\cdots, x_{m+1}),
  \end{eqnarray}
for all $f\in V^m_{\tilde{f}} =C_{\tilde{f}}^{m+1}(A;P),$ $g\in V^n_{\tilde{g}}=C_{\tilde{g}}^{n+1}(A;P)$.
The definition of $\circ_i$ can be extended to the case when $n=-1$ as follows
   \begin{eqnarray}
   && f\circ_ig(x_1,\cdots,x_i,x_{i+2},\cdots, x_{m+1}) \nonumber \\
     &=&(-1)^{\tilde{g}(x_1+\cdots+x_i)} f(x_1,\cdots,x_i,\phi g, x_{i+2},\cdots, x_{m+1})
  \end{eqnarray}
  One can readily verify that the condition \ref{LS10} holds and $\{V_m,\circ_i\}$ is a right pre-Lie supersystem. .
\end{exmpl}
\begin{defn}\label{rbSASdef1}
  Let $(V_m,\circ_i)$ be a pre-Lie supersystem. For every $m$ and $n$ we define a homogeneous $K$-bilinear  map  $\circ:V_m\times V_n\to V_{m+n}$ of degree $(0,0)$ by
  \begin{equation}\label{rbSAS1}
  f\circ g= \begin{cases}
   \sum_{i=0}^{m} (-1)^{ni}f\circ_i g , & \mbox{if } m\ge 0 \\
    0, & \mbox{if } m<0,
  \end{cases}
\end{equation}
 for all $f\in V_m$, $g\in V_n.$
\end{defn}
\begin{thm}\label{SASrbthm1}
  Let $(V_m,\circ_i)$ be a right pre-Lie supersystem. Then for $f\in V^m_{\tilde{f}}$, $g\in V^n_{\tilde{g}}$, $h\in V^p_{\tilde{h}}$
  \begin{itemize}
    \item[(a)] \[(f\circ g)\circ h-f\circ (g\circ h)=\sum_{\substack{ 0\le j\le i-1\\ n+i+1\le j\le m+n}}(-1)^{ni+pj}(f\circ_i g)\circ_jh\]
    \item[(b)] $(f\circ g)\circ h-f\circ(g\circ h)=(-1)^{np+\tilde{g}\tilde{h}}((f\circ h)\circ g-f\circ(h\circ g)).$
  \end{itemize}
\end{thm}
\begin{proof}
  By using Definitions \ref{rbSAS1}, \ref{LS10}
  \begin{eqnarray}\label{rbSAS2}
     && (f\circ g)\circ h- f\circ (g\circ h)\nonumber\\
     &=& \sum_{\substack{0\le j\le m+n\\0\le i\le m}}(-1)^{pj+ni}(f\circ_i g)\circ_j h- \sum_{\substack{0\le \lambda\le m\\0\le \mu\le n}}(-1)^{\lambda(n+p)+\mu p}f\circ_{\lambda} (g\circ_{\mu}h )\nonumber\\
     &=&   \sum_{\substack{0\le i\le m \\0\le j\le i-1\\n+i+1\le j\le m+n}}(-1)^{pj+ni}(f\circ_i g)\circ_j h+ \sum_{\substack{0\le i\le m \\i\le j\le n+i}}(-1)^{pj+ni}(f\circ_i g)\circ_j h\nonumber\\
     && -\sum_{\substack{0\le \lambda\le m\\0\le \mu\le n}}(-1)^{\lambda(n+p)+\mu p}f\circ_{\lambda} (g\circ_{\mu}h )\nonumber\\
      &=&   \sum_{\substack{0\le i\le m \\0\le j\le i-1\\n+i+1\le j\le m+n}}(-1)^{pj+ni}(f\circ_i g)\circ_j h+ \sum_{\substack{0\le i\le m \\i\le j\le n+i}}(-1)^{pj+ni}f\circ_i (g\circ_{j-i} h)\nonumber\\
     && -\sum_{\substack{0\le \lambda\le m\\0\le \mu\le n}}(-1)^{\lambda(n+p)+\mu p}f\circ_{\lambda} (g\circ_{\mu}h )\nonumber\\
     &=& \sum_{\substack{0\le i\le m \\0\le j\le i-1\\n+i+1\le j\le m+n}}(-1)^{pj+ni}(f\circ_i g)\circ_j h
  \end{eqnarray}
  From Equation \ref{rbSAS2},
  \begin{eqnarray}\label{rbSAS3}
     && (f\circ g)\circ h- f\circ (g\circ h)\nonumber\\
     &=& \sum_{\substack{0\le i\le m \\0\le j\le i-1\\n+i+1\le j\le m+n}}(-1)^{pj+ni}(f\circ_i g)\circ_j h \nonumber\\
     &=& (-1)^{pj+ni+\tilde{g}\tilde{h}}\Bigg\{\sum_{\substack{0\le i\le m \\0\le j\le i-1}}(f\circ_j h)\circ_{i+p} g +\sum_{\substack{0\le i\le m \\n+i+1\le j\le m+n}}(f\circ_{j-n} h)\circ_i g\Bigg \}\nonumber
\end{eqnarray}
Putting  $j=\lambda$, $i+p=\mu$ in the first sum  and $j-n=\lambda$, $i=\mu$ in the second sum of Equation \ref{rbSAS3}, we get
\begin{eqnarray}\label{rbSAS4}
   && (f\circ g)\circ h- f\circ (g\circ h)\nonumber \\
  &=& \sum_{\substack{0\le \lambda\le m \\\lambda+p+1\le\mu\le m+p}}(-1)^{p\lambda+n\mu+np+\tilde{g}\tilde{h}}(f\circ_\lambda h)\circ_\mu g\nonumber\\
  &&+\sum_{\substack{0\le \lambda\le m \\0\le \mu\le \lambda-1}}(-1)^{p\lambda+pn+n\mu+\tilde{g}\tilde{h}}(f\circ_\lambda h)\circ_\mu g\nonumber \\
  &=& (-1)^{np+gh}((f\circ h)\circ g- f\circ (h\circ g))
\end{eqnarray}
\end{proof}
\begin{cor}\label{SASrbcr1}
Let $\{ V_m, \circ_i\}$ be a right pre-Lie supersystem and let $V =\bigoplus V_m$ be the direct sum of the $\mathbb{Z}_2$-graded $K$-vector spaces $V_m=V^m_0\oplus V^m_1$.  Then with respect to multiplication given by $\circ$  in  Definition \ref{rbSASdef1}  $V$ is a $\mathbb{Z}$-graded right pre-Lie superalgebra.
\end{cor}
\begin{defn}\label{SASrbdef3}
 Let $\{V_m,\circ_i\}$ be a right pre-Lie supersystem and $W=\bigoplus_{(m,n)} W^m_n$ be  a $\mathbb{Z}\times \mathbb{Z}_2$-graded $K$-vector space.  We write $W_m=W^m_{0}\oplus W^m_{1}$ for all $m\in \mathbb{Z}$. We say that $W$ has the structure of a right supermodule over the right  pre-Lie supersystem $\{V_m,\circ_i\}$ if there exist homogeneous linear maps (for which we use the same notation $\circ_i$) from $W_m\times V_n$ to $W_{m+n}$ such that Equation \ref{LS10} holds for all $f\in W_m$, $g\in V_n$ and $h\in V_p.$ Now define $\circ :W\times V\to W$  by
\begin{equation}\label{rbSAS6}
        f\circ g=\begin{cases}
                \sum_{i=0}^{m} (-1)^{ni}f\circ_i g, & \mbox{if } m\ge 0 \\
                0, & \mbox{if } m<0,
              \end{cases}
\end{equation}
for all $f\in W_m$, $g\in V_n.$ Let $V=\bigoplus V_n$ be the $\mathbb{Z}$-graded right  pre-Lie superalgebra given by the right pre-Lie supersystem $\{V_m,\circ_i\}$.
Using similar arguments as in the proof of Theorem \ref{SASrbthm1} we conclude that $W$ is a right supermodule over the $\mathbb{Z}$-graded right  pre-Lie superalgebra $V$ with respect to the action  $\circ$ of $V$ on $W$ given by Equation \ref{rbSAS6}. Consider the $\mathbb{Z}$-graded Lie superalgebra structure given by the  $\mathbb{Z}$-graded right  pre-Lie superalgebra using Theorem \ref{SASrbthm2}.   If we define a right  action  of the $\mathbb{Z}$-graded Lie superalgebra $V$ on $W$ by  $[X,f]=X\circ f$, $X\in W$, $f\in V$, then $W$ is a right  supermodule over the $\mathbb{Z}$-graded Lie superalgebra $V.$
\end{defn}
As in Example \ref{SASrbe1}, if we consider only $\mathbb{Z}_2$-graded $K$-vector space structure of $A$ and put  $U=W=A$, $\phi=Id_A,$ then $\{C^m(A;A),\circ_i\}$ is a right pre-Lie supersystem where elements of $C^m(A;A)$ has degree $m-1$ and dimension $m.$ For $f\in C^m(A;P)$, $g\in C^n(A;A)$ we define $f\circ_i g\in C^{m+n-1}(A;P)$ by
\begin{eqnarray}
   && f\circ_ig(x_1,\cdots,x_i,y_1,\cdots,y_{n+1},x_{i+2},\cdots, x_{m+1}) \nonumber \\
   &=& (-1)^{\tilde{g}(x_1+\cdots+x_i)} f(x_1,\cdots,x_i,g(y_1,\cdots,y_{n+1}), x_{i+2},\cdots, x_{m+1}),
\end{eqnarray}
The definition of $\circ_i$ can be extended to the case when $n=-1$ as follows
\begin{eqnarray}
&& f\circ_ig(x_1,\cdots,x_i,x_{i+2},\cdots, x_{m+1}) \nonumber \\
&=&(-1)^{\tilde{g}(x_1+\cdots+x_i)} f(x_1,\cdots,x_i, g, x_{i+2},\cdots, x_{m+1})
\end{eqnarray}
We can easily verify that $C^*(A;P)$ is a right supermodule over the right pre-Lie supersystem $\{C^m(A;A),\circ_i\}$. From Corollary \ref{SASrbcr1},  $\{C^*(A;A),\circ\}$ is right pre-Lie superalgebra. By Theorem \ref{SASrbthm2}, $\{C^*(A;A),[-,-]\}$ is a $\mathbb{Z}$-graded Lie superalgebra. Also, from  Definition \ref{SASrbdef3}, we can see that $C^*(A;P)$ is a right  supermodule over the right pre-Lie superalgebra $\{C^*(A;A),\circ\}.$ For $f\in C^m_{\tilde{f}}(A;P)$, $g\in C^n_{\tilde{g}}(A;A)$, if we define $[f,g]=f\circ g$, and $[g,f]=-(-1)^{mn+\tilde{f}\tilde{g}}[f,g]$, then $C^*(A;P)$ is (two-sided) supermodule over the $\mathbb{Z}$-graded Lie superalgebra $\{C^*(A;A),[-,-]\}$.

We define $\pi\in C^2_0(A;A)$ by $\pi(a,b)=ab.$ We observe that $\delta \pi=0$ and $\delta Id_A=\pi,$ that is, $\pi\in B_0^2(A;A).$ By using direct definitions we have
\begin{equation}\label{SASrbe11}
f\cup g=(\pi\circ_0 f)\circ_m g,
\end{equation}
for all $f\in C^m(A;A)$ and $g\in C^n(A;A).$  Also, one can easily verify using definitions of $\pi$, $\delta$ and $\circ$ that
\begin{eqnarray}\label{SASrbe5}
  \delta f &=& -f\circ \pi +(-1)^{m-1}\pi\circ f\nonumber\\
  &=&(-1)^{m-1}(\pi\circ f-(-1)^{m-1}f\circ \pi ),
\end{eqnarray}
for all $f\in C^m_{\tilde{f}}(A;A)$.
By using Equation \ref{SASrbeqn120}, we have
\begin{eqnarray}\label{SASrbe6}
   \delta f  &=& [f,-\pi] \nonumber \\
   &=&(-1)^{m-1}[\pi,f].
\end{eqnarray}
From Equation \ref{SASrbe5} it is clear that $\delta$ is a right inner derivation of degree $(1,0)$ of the $\mathbb{Z}$-graded Lie superalgebra $\{C^*(A;A),[-,-]\}$ induced by $-\pi$.  Hence $\{C^*(A;A),[-,-]\}$ is a differential $\mathbb{Z}$-graded Lie superalgebra. Next, we have following  result:
\begin{thm}\label{SASrbthm3}
Let $A$ be a  associative superalgebra. Then
\begin{eqnarray}\label{SASrbeqn50}
f\circ \delta g-\delta(f\circ g)+(-1)^{n-1}\delta f\circ g  &=& (-1)^{n-1}\{(-1)^{\tilde{f}\tilde{g}}g\cup f-(-1)^{mn}f\cup g\},\nonumber\\
\end{eqnarray}
for all $f\in C^m_{\tilde{f}}(A;A)$ and $g\in C^n_{\tilde{g}}(A;A)$.
\end{thm}
\begin{proof}
Using the Equation \ref{SASrbe5} and Definition \ref{SASrbde10}, we have
\begin{eqnarray}\label{SASrbe8}
&&f\circ \delta g-\delta(f\circ g)+(-1)^{n-1}\delta f\circ g\nonumber\\
 &=& (-1)^{n-1}f\circ (\pi\circ g)-f\circ(g\circ \pi) -(-1)^{m+n-2}\pi\circ(f\circ g)+(f\circ g)\circ \pi\nonumber \\
   &&+ (-1)^{n-1+m-1}(\pi\circ f)\circ g-(-1)^{n-1}(f\circ \pi)\circ g \nonumber\\
   &=&(-1)^{m+n} \{(\pi\circ f)\circ g-\pi\circ(f\circ g)\} \nonumber\\
\end{eqnarray}
Using Theorem \ref{SASrbthm1} $(a)$ and  Equations \ref{SASrbe11}, \ref{LS12} we have
\begin{eqnarray}\label{SASrbe12}
(\pi\circ f)\circ g-\pi\circ(f\circ g)&=&\sum_{\substack{ 0\le j\le i-1\\ m+i\le j\le m}}(-1)^{(m-1)i+(n-1)j}(\pi\circ_i f)\circ_jg\nonumber\\
   &=& (-1)^{(n-1)m}(\pi\circ_0 f)\circ_m g+(-1)^{m-1}(\pi\circ_1 f)\circ_0 g\nonumber \\
   &=& (-1)^{(n-1)m}(\pi\circ_0 f)\circ_m g+(-1)^{m-1+\tilde{f}\tilde{g}}(\pi\circ_0 g)\circ_n f \nonumber \\
   &=&(-1)^{(n-1)m}f\cup g+(-1)^{m-1+\tilde{f}\tilde{g}}g\cup f\nonumber\\
   &=&(-1)^{m-1}\{(-1)^{\tilde{f}\tilde{g}}g\cup f- (-1)^{mn}f\cup g \}.
\end{eqnarray}
Using Equations \ref{SASrbe8},\ref{SASrbe12} we conclude Equation \ref{SASrbeqn50}. This completes  proof of the theorem.
\end{proof}
\begin{cor}
If $A$ is an associative superalgebra, then the  $\mathbb{Z}$-graded associative superalgebra $\{H^*(A;A),\cup\}$ is commutative, that is, if $u\in H^m_{\tilde{u}}(A;A)$, $v\in H^n_{\tilde{v}}(A;A)$, then $$u\cup v=(-1)^{mn+\tilde{u}\tilde{v}}v\cup u.$$
\end{cor}
\begin{proof}
By using Theorem \ref{SASrbthm3}, for $f\in Z^m_{\tilde{f}}(A;A)$, $g\in Z^n_{\tilde{g}}(A;A)$ we have
  $$\{(-1)^{\tilde{f}\tilde{g}}g\cup f- (-1)^{mn}f\cup g = (-1)^{n}\delta (f\circ g).$$
Hence we conclude that $u\cup v=(-1)^{mn+\tilde{u}\tilde{v}}v\cup u,$ for all  $u\in H^m_{\tilde{u}}(A;A)$, $v\in H^n_{\tilde{v}}(A;A)$.
\end{proof}
If $A$ is an associative superalgebra and $P$ is a (two sided ) supermodule over $A,$ then $A\oplus P$ has the structure of an associative superalgebra with respect to the multiplication given by $$(a,x).(b,y)=(a.b,a.y+x.b),$$ for all  $a,b\in A$, $x,y\in P.$ Consider the natural inclusions $C^n(A;A),C^n(A;P)\subset C^n(A\oplus P;A\oplus P)$ defined by $f(a_1,\cdots, a_n)=0$, if some $a_i$ is in $P$,for all $f$ in $C^n(A;A)$ or $C^n(A;P)$.  Now, Theorem \ref{SASrbthm3} holds for the  associative superalgebra $A\oplus P$.  Observe that if $f\in C^m(A;A)$, $g\in C^n(A;P)$ are cocyles, then $g\in C^n(A\oplus P;A\oplus P)$ is a cocycle but $f\in C^n(A\oplus P;A\oplus P)$ need not be a cocycle.  Also, observe that if  $f\in Z^m(A;A)$, $g\in Z^n(A;P)$, then in the Equation \ref{SASrbeqn50} in the Theorem \ref{SASrbthm3} for the associative superalgebra $A\oplus P$, $(\delta f)\circ g=0,$ $f\circ (\delta g)=0.$ This implies that   if  $f\in Z^m(A;A)$, $g\in Z^n(A;P)$, then
$$\delta(f\circ g)=- (-1)^{n-1}\{(-1)^{\tilde{f}\tilde{g}}g\cup f-(-1)^{mn}f\cup g\}.$$  Hence we have following result:
\begin{cor}
Let $A$ be a $\mathbb{Z}$-graded associative superalgebra and $P$ be a supermodule over $A.$ Then $$u\cup v=(-1)^{mn+\tilde{u}\tilde{v}}v\cup u, $$
if  $u\in H^m_{\tilde{u}}(A;A)$, $v\in H^n_{\tilde{v}}(A;P)$.
\end{cor}
\begin{thm}
Let $A$ be a  associative superalgebra and $P$ be a (two sided) supermodule over $A.$  Then $C^*(A;P)$ is a (two sided) supermodule over the $\mathbb{Z}$-graded Lie superalgebra $\{C^*(A;A),[-,-]\}$. Also, we have
$$[Z^{m_1}_{n_1}(A;P),Z^{m_2}_{n_2}(A;A)]\subset Z^{m_1+m_2}_{n_1+n_2}(A;P);$$
$$[B^{m_1}_{n_1}(A;P),Z^{m_2}_{n_2}(A;A)],[Z^{m_1}_{n_1}(A;P),Z^{m_2}_{n_2}(A;A)]\subset B^{m_1+m_2}_{n_1+n_2}(A;P),$$
$\{H^*(A;A),[-,-]\}$ is a  $\mathbb{Z}$-graded Lie superalgebra and  $H^*(A;P)$ is a (two sided) supermodule over $\{H^*(A;A),[-,-]\}$.
\end{thm}
\begin{proof}
From our previous discussion we already know that $C^*(A;P)$ is a supermodule over the $\mathbb{Z}$-graded Lie superalgebra $\{C^*(A;A),[-,-]\}$. Consider the  associative superalgebra $A\oplus P$ and $\mathbb{Z}$-graded Lie superalgebra $\{C^*(A\oplus P;A\oplus P),[-,-]\}$.
Since coboundary map $\delta$ is an inner right derivation  of $\{C^*(A\oplus P;A\oplus P),[-,-]\}$ induced by $-\pi$, we have
\begin{eqnarray}\label{SASrbeqn51}
     \delta[f,g]&=&[f,\delta g]+(-1)^{n-1}[\delta f,g]
  \end{eqnarray}
 for all $f\in C^m(A\oplus P;A\oplus P)$, $g\in C^n(A\oplus P;A\oplus P)$. In particular, Equation \ref{SASrbeqn51} holds for $f $ in $C^m(A;A)$ or $C^m(A;P)$, $g$ in $C^n(A;P)$ or $C^n(A;A)$; and for $\delta$ as a coboundary map on $C^*(A;A)$ or $C^*(A;P)$. From this we conclude the theorem.
\end{proof}

\section{Formal Deformation of Associative Superalgebras }\label{rbsec20}
 Given an  associative superalgebra $A=A_0\oplus A_1$, we denote the ring of all formal power series with coefficients in $A$ by $A[[t]]$. Thus  $A[[t]]=A_0[[t]]\oplus A_1[[t]]$. If  $a_t\in A[[t]]$, then $a_t=a_{t_0}\oplus a_{t_1}$, where $a_{t_0}\in A_0[[t]]$ and $a_{t_1}\in A_1[[t]]$. $K[[t]]$ denotes the  ring of all formal power series with coefficients in $K$.

\begin{defn}\label{rb2}
 A formal one-parameter deformation of an associative  superalgebra $A=A_0\oplus A_1$ is a $K[[t]]$-bilinear map  $$\mu_t : A[[t]]\times A[[t]]\to A[[t]]$$ satisfying the following properties:
\begin{itemize}
  \item[(a)]  $\mu_t(a,b)=\sum_{i=0}^{\infty}\mu_i(a,b) t^i$, for all  $a,b\in A$, where $\mu_i:A\times A\to A$, $i\ge 0$ are  bilinear homogeneous mappings of degree zero  and $\mu_0(a,b)=\mu( a,b)$ is the original  product on $A$.
\item[(b)]\begin{equation}\label{DLT1}
   \mu_t( \mu_t(a,b),c)=\mu_t(a,\mu_t(b,c)),
  \end{equation}
  for all homogeneous $a,b,c\in A$.
\end{itemize}
The Equation  \ref{DLT1} is equivalent to following equation:
   \begin{eqnarray}\label{rbeqn1}
 \sum_{i+j=r} \mu_i( \mu_j(a,b),c)&=& \sum_{i+j=r}\{\mu_i(a,\mu_j(b,c)),
   \end{eqnarray}
for all homogeneous $a,b,c\in A$.
\end{defn}
Next we give definition of a formal deformation of finite order of an associative superalgebra $A$.
\begin{defn}\label{rb3}
  A  formal one-parameter deformation of order $n$  of an associative superalgebra  $A=A_0\oplus A_1$ is a $K[[t]]$-bilinear map  $$\mu_t : A[[t]]\times A[[t]]\to A[[t]]$$ satisfying the following properties:
\begin{itemize}
  \item[(a)]  $\mu_t(a,b)=\sum_{i=0}^{n}\mu_i(a,b) t^i$,  $\forall a,b,c\in A$, where $\mu_i:A\times A\to A$, $0\le i\le n$, are $K$-bilinear homogeneous maps of degree $0$, and $\mu_0$ is the original product on $A$.
      \item[(b)] \begin{equation}\label{FDLS}
   \mu_t( \mu_t(a,b),c)=\mu_t(a,\mu_t(b,c)),
  \end{equation}
  for all homogeneous $a,b,c\in A$.
\end{itemize}
\end{defn}
\begin{rem}\label{rbrem1}
  \begin{enumerate}
    \item For $r=0$, conditions \ref{rbeqn1} is equivalent to the fact that $A$ is an associative superalgebra.
    \item For $r=1$, conditions \ref{rbeqn1} is equivalent to
     \begin{eqnarray*}\label{rrbeqn1}
      0&=&-\mu_1(\mu_0(a,b),c)-\mu_0(\mu_1(a,b),c)\\
       &&+ \mu_1(a,\mu_0(b,c))+ \mu_0(a,\mu_1(b,c))\\
       &=&\delta^2\mu_1(a,b,c); \;\text{for all homogeneous}\; a,b,c\in A.
   \end{eqnarray*}
Thus for $r=1$,  \ref{rbeqn1} is equivalent to saying that $\mu_1\in C_0^2(A;A)$  is a cocycle. In general, for $r\ge 0$, $\mu_r$ is just a 2-cochain, that is,  $\mu_r\in C_0^2(A;A).$
    \end{enumerate}
\end{rem}

\begin{defn}
 We call the cochain  $\mu_1 \in C_0^2(A;A)$  infinitesimal of the   deformation $\mu_t$. In general, if $\mu_i=0,$ for $1\le i\le n-1$, and $\mu_n$ is a nonzero cochain in  $C_0^2(A;A)$, then we call  $\mu_n$  n-infinitesimal of the  deformation $\mu_t$.
\end{defn}
\begin{prop}
  The infinitesimal   $\mu_1\in C_0^2(A;A)$ of the  deformation  $\mu_t$ is a cocycle. In general, $n$-infinitesimal  $\mu_n$ is a cocycle in $C_0^2(A;A).$
\end{prop}
\begin{proof}
  For n=1, proof is obvious from the Remark \ref{rbrem1}. For $n>1$, proof is similar.
\end{proof}
From Equation \ref{rbeqn1}, we have
\begin{eqnarray}\label{rbeqn200}
&&\sum_{\substack{ i+j=r\\i,j>0}} \mu_i( \mu_j(a,b),c)- \sum_{\substack{ i+j=r\\i,j>0}}\{\mu_i(a,\mu_j(b,c))\notag\\
&=&\mu_r(a,\mu_0(b,c))+ \mu_0(a,\mu_r(b,c)) -\mu_0(\mu_r(a,b),c)- \mu_r(\mu_0(a,b),c)\notag\\
&=&\delta^2\mu_r(a,b,c)
\end{eqnarray}
\begin{defn}
Given a formal deformation $\mu_t$ of order $n$ of  an associative superalgebra $A=A_0\oplus A_1$,   we define a $3$-cochain $Ob_{n+1}(A)$ by 
$$Ob_{n+1}(A)=\sum_{\substack{ i+j=r\\i,j>0}} \mu_i( \mu_j(a,b),c)- \sum_{\substack{ i+j=r\\i,j>0}}\{\mu_i(a,\mu_j(b,c)).$$
We call $Ob_{n+1}(A)$  $(n+1)$th obstruction cochain for extending $\mu_t$ to a deformation of $A$ of order $n+1$.  
\end{defn}
As an application of Theorem \ref{SASrbthm1} we conclude following Lemma:
\begin{lem}\label{rblm210}
For all $\mu_\alpha,\mu_\beta,\mu_\gamma\in C^n_{0}(A;A)$ with $\beta,\gamma$ an even integer, we have
\begin{itemize}
\item[(a)] $\mu_\alpha\circ(\mu_\beta\circ \mu_\beta)=(\mu_\alpha\circ\mu_\beta)\circ \mu_\beta,$
\item [(b)] $(\mu_\alpha\circ\mu_\beta)\circ \mu_\gamma-\mu_\alpha\circ(\mu_\beta\circ \mu_\gamma)=-(\mu_\alpha\circ\mu_\gamma)\circ \mu_\beta+\mu_\alpha\circ(\mu_\gamma\circ \mu_\beta).$
\end{itemize}

\end{lem}
\begin{thm}
The $(n+1)$th obstruction cochain is a $3$-cocycle.
\end{thm}
\begin{proof}
Consider the right pre-Lie superalgebra $\{c*(A,A),\circ\}$.  By definition of $\circ,$ we have 
\begin{equation}\label{rbeqn201}
Ob_{n+1}(A)(a,b,c)=\sum_{\substack{ i+j=n+1\\i,j>0}} \mu_i\circ \mu_j(a,b,c).
\end{equation} 
By using Theorem \ref{SASrbthm3} we have 
\begin{eqnarray}\label{rbeqn202}
\mu_i\circ \delta \mu_j-\delta(\mu_i\circ \mu_j)-\delta \mu_i\circ \mu_j  &=&- \mu_j\cup \mu_i +\mu_i\cup \mu_j,
\end{eqnarray}
for all $\mu_i,\mu_j\in C^2_{0}(A;A)$.
From relations \ref{rbeqn201} and  \ref{rbeqn202} we have 
\begin{eqnarray}\label{rbeqn203}
\delta Ob_{n+1}(A)&=&\sum_{\substack{ i+j=n+1\\i,j>0}}\delta( \mu_i\circ \mu_j)\nonumber\\
 &=& \sum_{\substack{ i+j=n+1\\i,j>0}}\{\mu_i\circ \delta \mu_j-\delta \mu_i\circ \mu_j  + \mu_j\cup \mu_i -\mu_i\cup \mu_j\}\nonumber\\
 &=& \sum_{\substack{ i+j=n+1\\i,j>0}}\{\mu_i\circ \delta \mu_j-\delta \mu_i\circ \mu_j \}
 \end{eqnarray}
 Observe that if $\mu_t=\sum_{i=0}^{n}\mu_i t^i$ is a deformation of $A$ of order $n$, then 
 \begin{equation}\label{rbeqn204}
\delta \mu_{\gamma}(a,b,c)=\sum_{\substack{ \alpha+\beta=\gamma\\\alpha,\beta>0}} \mu_\alpha\circ \mu_\beta(a,b,c),
 \end{equation}
$ \forall \gamma\le n.$
 Using Equations  \ref{rbeqn203} and \ref{rbeqn204} we have 
 \begin{eqnarray}
 \delta Ob_{n+1}(A)&=&  \sum_{\substack{ i+j+k=n+1\\i,j,k>0}}\{\mu_i\circ  (\mu_j \circ \mu_k ) - (\mu_i\circ mu_j)\circ \mu_k\} \notag\\
 &=& \sum_{\substack{ i+j+k=n+1\\i,j,k>0,j<k}}\{\mu_i\circ  (\mu_j \circ \mu_k+\mu_k\circ\mu_j )\notag\\
 && -  (\mu_i\circ \mu_j)\circ \mu_k +(\mu_i\circ\mu_k)\circ \mu_j\}\;\;\;\;\text{(Using Lemma \ref{rblm210} Part (a))}\notag\\
 &=&0\;\;\;\;\text{(Using Lemma \ref{rblm210} Part (b))} \notag
 \end{eqnarray}
 
\end{proof}
As a consequence of above theorem we conclude following corollary
\begin{cor}
If $H^3(A;A)=0,$ then every $2$-cocycle in $C^2_0(A;A)$ is an infinitesimal of some deformation of $A.$
\end{cor}
\begin{defn}
Let   $\mu_t$  and $\tilde{\mu_t}$ be two formal deformations of an associative  superalgebra $A=A_0\oplus A_1$. A formal isomorphism from the deformation $\mu_t$ to $\tilde{\mu_t}$  is a $K[[t]]$-linear automorphism $\Psi_t:A[[t]]\to A[[t]]$ given by  $$\Psi_t=\sum_{i=0}^{\infty}\psi_it^i,$$ where each $\psi_i$ is a homogeneous $K$-linear map $A\to A$ of degree $0$, $\psi_0(a)=a$, for all $a\in A$ and $$\tilde{\mu_t}(\Psi_t(a),\Psi_t(b))=\Psi_t\circ\mu_t(a,b),$$ for all $a,b\in A.$\\
 We call two  deformations $\mu_t$  and $\tilde{\mu_t}$ of an associative  superalgebra $A$ to be equivalent if there exists a formal isomorphism  $\Psi_t$ from $\mu_t$ to  $\tilde{\mu_t}$.
Observe that  Formal isomorphism on the collection of all  formal deformations of an associative  superalgebra $A$ is an equivalence relation.
 We call  a formal deformation of $A$ that is equivalent to the deformation $\mu_0$  a trivial deformation.
\end{defn}
\begin{thm}
 The cohomology class of the infinitesimal of a  deformation $\mu_t$ of   an associative  superalgebra $A$ is same for  each member of  equivalence class of $\mu_t$.
\end{thm}
\begin{proof}
If   $\Psi_t$ is a formal  isomorphism  from  $\mu_t$ to  $\tilde{\mu_t}$, the   we have, for  all $a,b\in A$,  $\tilde{\mu_t}(\Psi_ta,\Psi_tb)=\Psi_t\circ \mu_t(a,b)$. In particular, we have \begin{eqnarray*}
     (\mu_1-\tilde{\mu_1})(a,b)&=& \mu_0(\psi_1a,b)+\mu_0( a,\psi_1b)-\psi_1(\mu_0( a,b))\\
     &=& \delta^1\psi_1(a,b).
\end{eqnarray*}
Thus we have $\mu_1-\tilde{\mu_1}=\delta^1\psi_1$. This completes the proof.
\end{proof}

\bibliographystyle{alpha}
\bibliography{cohom-deform-of-ass-superalgebra}

\begin{thebibliography}{10}
\expandafter\ifx\csname url\endcsname\relax
  \def\url#1{\texttt{#1}}\fi
\expandafter\ifx\csname urlprefix\endcsname\relax\def\urlprefix{URL }\fi
\expandafter\ifx\csname href\endcsname\relax
  \def\href#1#2{#2} \def\path#1{#1}\fi

\bibitem{MR11076}
G.~Hochschild, On the cohomology groups of an associative algebra, Ann. of
  Math. (2) 46 (1945) 58--67.

\bibitem{MR16762}
G.~Hochschild, On the cohomology theory for associative algebras, Ann. of Math.
  (2) 47 (1946) 568--579.

\bibitem{MR22842}
G.~Hochschild, Cohomology and representations of associative algebras, Duke
  Math. J. 14 (1947) 921--948.

\bibitem{MR161898}
M.~Gerstenhaber, The cohomology structure of an associative ring, Ann. of Math.
  (2) 78 (1963) 267--288.

\bibitem{MR171807}
M.~Gerstenhaber, On the deformation of rings and algebras, Ann. of Math. (2) 79
  (1964) 59--103.

\bibitem{MR0207793}
M.~Gerstenhaber, On the deformation of rings and algebras. {II}, Ann. of Math.
  84 (1966) 1--19.

\bibitem{MR240167}
M.~Gerstenhaber, On the deformation of rings and algebras. {III}, Ann. of Math.
  (2) 88 (1968) 1--34.

\bibitem{MR4278718}
A.~Ben~Hassine, L.~Chen, C.~Sun, Representations and one-parameter formal
  deformations of {B}i{H}om-{N}ovikov superalgebras, Rocky Mountain J. Math.
  51~(2) (2021) 423--438.

\bibitem{MR2654093}
Y.~Khakimdjanov, R.~M. Navarro, Deformations of filiform {L}ie algebras and
  superalgebras, J. Geom. Phys. 60~(9) (2010) 1156--1169.

\bibitem{MR4117244}
S.~Guo, X.~Zhang, S.~Wang, Representations and deformations of
  {H}om-{L}ie-{Y}amaguti superalgebras, Adv. Math. Phys. (2020) Art. ID
  9876738, 12.

\bibitem{MR2761298}
S.~Bonanos, J.~Gomis, K.~Kamimura, J.~Lukierski, Deformations of {M}axwell
  superalgebras and their applications, J. Math. Phys. 51~(10) (2010) 102301,
  26.

\bibitem{MR871615}
B.~Binegar, Cohomology and deformations of {L}ie superalgebras, Lett. Math.
  Phys. 12~(4) (1986) 301--308.

\bibitem{MR0438925}
L.~Corwin, Y.~Ne'eman, S.~Sternberg, Graded {L}ie algebras in mathematics and
  physics ({B}ose-{F}ermi symmetry), Rev. Modern Phys. 47 (1975) 573--603.

\end{thebibliography}

\end{document}